\documentclass[11pt]{article}
\usepackage[cp1251]{inputenc}
\usepackage{amsfonts}
\usepackage{amssymb}
\usepackage{amsthm}
\usepackage{amsmath}
\usepackage{xcolor}
\usepackage[english]{babel}
\newcommand{\be}{\begin{equation}}
\newcommand{\ee}{\end{equation}}

\newtheorem{theorem}{Theorem}[section]
\newtheorem{lemma}{Lemma}[section]

\newtheorem{proposition}{Proposition}[section]

\newtheorem{remark}{Remark}[section]

\newtheorem{com}{Comment}[section]

\begin{document}

\title{Markov branching process with infinite variance and non-homogeneous immigration with infinite mean}
\author{Kosto V. Mitov, \\
Faculty of Pharmacy, Medical University, Pleven, Bulgaria \\
email: kmitov@yahoo.com \\
Nikolay M. Yanev, \\
Institute of Mathematics and Informatics, BAS, Sofia, Bulgaria, \\
email: yanev@math.bas.bg} \maketitle

\begin{abstract}
The paper studies a class of critical Markov branching processes
with infinite variance of the offspring distribution. The processes
admit also an immigration component at the jump-points of a
non-homogeneous Poisson process, assuming that the mean number of
immigrants is infinite and the intensity of the Poisson process
converges to zero. The asymptotic behavior of the probability for
non-visiting zero is obtained. Limiting distributions are proved,
under suitable normalization of the sample paths, depending on the
offspring distribution, on the distribution of the immigrants and on
the intensity of the Poisson process.
\end{abstract}




\section{Introduction}

The paper deals with Markov branching processes with immigration in
time-moments generated by Poisson measure with a local intensity
$r(t).$ We consider the critical case when the offspring mean is
equal to one, but the offspring variance is infinite. The
distribution of immigrants belongs to the class of stable laws with
infinite mean and $r(t)$ is a regularly varying function (r.v.f.)
converging to zero. Then the considered branching processes are
non-homogeneous in time.

Recall that the first branching process with immigration was
formulated by Sevastyanov \cite{Sevastyanov57}. He investigated a
single-type Markov branching process in which immigration occurs
according to a time homogeneous Poisson process, and proved limiting
distributions. Branching processes with time non-homogeneous
immigration were first proposed by Durham \cite{Durham} and Foster
and Williamson \cite{fowi}. Further results
can be found in Badalbaev and Rahimov \cite{BadRah} and Rahimov \cite%
{Rahimov}. See also a review paper of Rahimov \cite{rah2021}. A
model with
critical non-homogeneous migration was investigated by Yanev and Mitov \cite%
{YanMit85}. Critical Sevastyanov branching processes with
non-homogeneous immigration were studied in\ \cite{mityan} and
critical multitype Markov branching processes with non-homogeneous
Poisson immigration were considered by Mitov et al. \cite{myh}.
Notice that the limiting distributions in these models were obtained
in the case of finite first and second offspring characteristics as
well as those of the immigration components.

The asymptotic behaviour of branching processes is quite different
in the case of finite or infinite offspring variance. Zolotarev
\cite{zolotarev} was first who obtained limiting distributions for
Markov branching processes with infinite offspring variance. Pakes
\cite{pakes1975}, \cite{pakes2010} investigated respectively
Bienaym\'{e}-Galton-Watson process and Markov branching process in
the critical case with infinite offspring variance and finite mean
of the immigrants, where for the continuous time case it is assumed
that the immigration occurs at time-points of a homogeneous Poisson
process. Imomov and Tukhtaev \cite{Imomov} considered critical Bienaym\'{e}%
-Galton-Watson process with infinite offspring variance and infinite
mean of immigrants and extended also some of the results of Pakes
\cite{pakes1975}. Sagitov \cite{sagitov} studied multi-type Markov
branching processes in the case of homogeneous Poisson immigration
with infinite second moments of the offspring distributions and
infinite means of the immigrants.

Branching processes with time non-homogeneous immigration find
applications for investigating the dynamics of biological systems,
particularly cellular populations (see, for example, \cite{yay1,
Hyr1, Hyr2}). In these applications, the stem cells often are
considered as an immigration component.

We have to mention that some of the results obtained here are
similar to
some of the results obtained in the discrete time case by Rahimov \cite%
{rah86,rah93mn} for Bienaym\'{e}-Galton-Watson branching processes
and this is not surprising. Let us note also that the methods of
studying in the present work are based on the functional equations
for the probability generating functions, stationary measures and
some other methods which essentially differ from the methods used in
\cite{rah86,rah93mn}.

A detailed description of the considered model is presented in
Section 2. Some preliminary results and basic assumptions are given
in Section 3. The asymptotic behavior of the $P_{t}$\ (probability
for non-visiting zero state at moment $t$) is investigated in
Section 4 (Theorems 4.1-4.3). Thus Theorem
4.1, Theorem 4.2 and Theorem 4.3-(i) establish the conditions under which $%
P_{t}$ converges to zero as $t\rightarrow \infty $ with various
rates
depending of the parameters of reproduction ($\gamma $), of immigration ($%
\alpha $) and of Poisson measure ($\theta $). Under the conditions
of Theorem 4.3-(ii) $P_{t}$ converges to a positive probability less
than 1 which is exactly calculated. Finally, at the conditions of
Theorem 4.3-(iii) we obtain that $P_{t}\rightarrow 1$ as
$t\rightarrow \infty .$

Under the same basic conditions various types of limiting
distributions are obtained in five theorems presented in Section 5.
In fact we proved eight different type limiting distributions. Note
that seven of them are under the condition for non-visiting zero
state. Thus in Theorems 5.1-(i), 5.2-(ii) and 5.4-(ii) under
suitable additional conditions we proved a stationary discrete time
limiting distribution and we obtained an integral form for the
corresponding probability generating function (p.g.f.) depending
only from the local characteristics of the process (offspring p.g.f.
$f(s)$ and the p.g.f. of the immigrants $g(s),$ see Remark 5.1). For
the other limiting distributions under a normalization with a
suitable r.v.f. we obtained the corresponding Laplace transforms.
The most interesting results are given in Theorem 5.1-(iii) where we
have two singular to each other conditional limiting distributions,
where the first one is a non-proper stationary distribution with an
atom at infinity and the second one under a suitable normalization
has an atom at zero. Finally in Theorem 5.5-(iii) we obtained non
conditional limiting distribution under a suitable normalizing
r.v.f. and the limiting random variable is just stable with
parameter $\alpha $ (from the p.g.f. of the immigrants). Comments
with discussion of the results are given after the proofs of all
theorems in Sections 4 and 5.


\section{Description of the models}

\label{sec2}

A Markov branching process can be described as follows. The
particles of a given type evolve independently of each other, lives
random time $\tau $ with distribution function (d.f.)
$G(t)=\mathbf{P}\left\{ \tau \leq t\right\} =1-e^{-\mu t}$, $t\geq
0,\mu >0,$ and at the end of its life the particle produces random
number $\xi \geq 0$ of new particles of the same type. The number of
particles $Z(t)$ at moment $t\geq 0 $ is known as Markov branching
process (see \cite{an}, \cite{haris}). Denote by
$h(s)=\mathbf{E}\left[ s^{\xi }\right] $ the offspring p.g.f. and
$F(t;s)=\mathbf{E}\left[ s^{Z(t)}|Z(0)=1\right] ,\ t\geq 0,\ \ s\in
[0,1].$ It is well known that (see e.g. \cite{an}, \cite{haris})
\begin{eqnarray}
&& \frac{\partial F(t;s)}{\partial t}=\mu [h\left(F(t;s)\right)
-F(t;s)],  \label{kolmo0}
\end{eqnarray}
Let $\left( S_{k},I_{k}\right) $, $k=0,1,2,\ldots $, be independent of $Z(t)$%
, where $0=S_{0}<S_{1}<S_{2}<\cdots $are jump points of an
non-homogeneous Poisson process $\nu (t)$ and the random variables
$\{I_{k}\}$\ are independent, identically distributed (i.i.d.) with
non-negative integer
values. Denote by $r(t)$ the intensity of $\nu (t)$ with a mean measure $%
\displaystyle R(t)=\int_{0}^{t}r(u)du$. Let $\displaystyle g(s)=\mathbf{E}%
\left[ s^{I_{k}}\right] $ be the p.g.f. of the immigrants.

Assume that at every jump-point $S_{k}$, a random number $I_{k}$ of
new particles immigrate into the process $Z(t)$ and they participate
in the evolution as the other particles. Let us denote the new
process by $Y(t)$. It can be strictly defined as follows
\begin{eqnarray*}
\displaystyle Y(t)=\sum_{k=1}^{\nu
(t)}\sum_{j=1}^{I_{k}}Z^{(k,j)}\left( t-S_{k}\right) ,\
\displaystyle t\geq 0,
\end{eqnarray*}%
where $\left\{ Z^{(k,j)}(t)\right\} $ are i.i.d. copies of $Z(t)$.
The process $Y(t)$, $t\geq 0,$ is called Markov branching process
with non-homogeneous Poisson immigration (MBPNPI).

For $\Phi (t;s):=\mathbf{E}\left[ s^{Y(t)}|Y(t)=0\right] $ we have
the following equation
\begin{eqnarray}
&& \Phi(t;s)=\exp\left\{-\int_0^t r(t-u)(1 - g(F(u;s)))du \right\},
\ \Phi(t;0)=1.  \label{yanev}
\end{eqnarray}
The proof is given in \cite{yay1} and in the more general multitype
case in \cite{myh}.

For the intensity of the Poisson process, we assume additionally the
following condition
\begin{eqnarray}  \label{rt-0}
\displaystyle r(t)=t^{-\theta}L_R(t), \ \mathrm{\ where \ } \ \
\theta >0,
\end{eqnarray}
and $L_R(.)$ is a slowly varying function (s.v.f.) at infinity.


\section{Basic assumptions and preliminary results}

\label{sec3}

For the branching mechanism we assume that the offspring p.g.f.
$f(s)$ has the following representation
\begin{eqnarray}  \label{infinite-var}
&& f(s)=s+(1-s)^{\gamma+1}L\left(\frac{1}{1-s}\right), \ \ \ s \in
[0,1),
\end{eqnarray}
where $\gamma \in (0,1]$ and $L(.)$ is a s.v.f. at infinity. Thus,
the process $Z(t),t\geq 0,$ is critical. If $\gamma <1$ the
offspring variance in infinite.

\begin{com}
\label{com1} If $\gamma =1$ and $L(t)\rightarrow b$ then the
offspring variance is finite. The results for this case follow
directly from the corresponding results for the multitype Markov
processes with non-homogeneous Poisson immigration studied in
\cite{myh}. If $\gamma =1$ the offspring variance can also be
infinite, depending on the properties of the slowly varying function
$L(.)$.
\end{com}

It is known (see \cite{haris}, Theorem 12.1) that a critical Markov
branching process has an invariant measure whose p.g.f. $U(s)$ is
given by $\displaystyle U(s)=\int_{0}^{s}\frac{du}{f(u)-u},\ \ \ \
0\leq s\leq 1.$ The Kolmogorov backward equation (\ref{kolmo0}) can
be written as follows $\displaystyle
\int_{s}^{F(t;s)}\frac{du}{f(u)-u}=\mu t.$ This leads to
$U(F(t;s))=U(s)+\mu t$. Denote by
\begin{eqnarray*}
V(x)=U\left( 1-\frac{1}{x}\right) =\int_{0}^{1-1/x}\frac{du}{f(u)-u}%
=\int_{1}^{x}\frac{u^{\gamma -1}}{L(u)}du,x\geq 1.
\end{eqnarray*}
So $V(x)$ is regularly varying with exponent $\gamma $. Then its
inverse $W(y)$ is regularly varying with exponent $1/\gamma .$ Let
us note that $V(x)$ and  $W(y)$ are increasing (see e.g.
\cite{pakes1975},\cite{pakes2010}). Using the above relations we get
for $\ s \in [0,1)$
\begin{eqnarray}
&& 1/[1-F(t;s)]=W\left( \mu t+V(1/(1-s)\right) ), \ \ s \in [0,1).
\label{q-asymp-2-s}
\end{eqnarray}
Substituting $s=0$ we have $1-F(t;0)=1/W(\mu t).$ For the p.g.f. of
the immigrants we will assume that
\begin{eqnarray}
&& g(s)=1-(1-s)^\alpha l\left(\frac{1}{1-s}\right), \ \ s \in [0,1]
\label{im-fin}
\end{eqnarray}
where $\alpha \in (0,1]$ and $l(x)$ is a s.v.f. at infinity.

\begin{com}
\label{com2} If $\alpha \in (0,1)$ the mean number of immigrants is
infinite. In the case when $\alpha=1$ the mean number of immigrants
can be infinite or finite depending on the s.v.f. $l(.)$. If
$\alpha=1$, and $l(x) \to m \in (0,\infty),$ then
$\mathbf{E}[I_k]=m$ is finite.
\end{com}

Let us denote $\displaystyle \Psi
(x)=1/[1-g(1-\frac{1}{x})]=\frac{x^{\alpha }}{l(x)},\ x\geq 1.$  The
function $\Psi (.)$ is non decreasing in $[1,\infty )$ and let
$\displaystyle \overleftarrow{\Psi }(x),\ x\geq 1,$ be its inverse
function which is also non-decreasing in $[1,\infty )$. Then $g(s)$
can be written in the following form
\begin{eqnarray*}
g(s)=1-1/\Psi \left( 1/(1-s)\right) ,\ \ s \in [0,1].
\end{eqnarray*}
Further for convenience we will denote (see also
(\ref{q-asymp-2-s}))
\begin{eqnarray}
&&q(t;s):=1-g(F(t;s))=(1-F(t;s))^{\alpha }l\left( 1/[1-F(t;s)]\right)  \notag \\
&=& 1/\Psi \left( W\left( \mu t+V\left( 1/(1-s)\right) \right)
\right).\label{qts}
\end{eqnarray}
From (\ref{qts}) with $s=0$ it follows that as $t\rightarrow \infty
$,
\begin{eqnarray}
\label{gFt} q(t):=q(t;0)= \frac{1}{\Psi\left( W\left( \mu
t\right)\right)} \sim t^{-\alpha/\gamma}L_Q(t),
\end{eqnarray}
where $L_Q(.)$ is s.v.f. at infinity.

We will use the following notations
\begin{eqnarray*}
Q(t):=\int_{0}^{t}q(u)du,\ Q:=\int_{0}^{\infty }q(u)du,\
\Delta(s):=\int_{0}^{\infty }q(u;s)du, s\in [0,1],
\end{eqnarray*}
when the last two integrals converge.
\begin{proposition} \label{proposition} The following representations hold
\begin{eqnarray}
&& \label{QDeltas} \Delta (s)=\int_{s}^{1}\frac{1-g(u)}{\mu
(f(u)-u)}du, \ Q=\int_{0}^{1}\frac{1-g(u)}{\mu (f(u)-u)}du.
\end{eqnarray}
\end{proposition}
\begin{proof} Since $q(t;s)$ is non-increasing in $s\in [0,1]$ then the
convergence of $\displaystyle Q=\int_{0}^{\infty }q(t)dt$ leads to
the uniform convergence of $\displaystyle\Delta (s)=\int_{0}^{\infty
}q(t;s)dt$ on $[0,1]$. Then $\displaystyle \frac{d}{ds}\Delta
(s)=-\int_{0}^{\infty }\frac{d}{dF}g(F(t;s))\frac{\partial
}{\partial s}F(t;s)dt.$ Note that by the forward Kolmogorov equation
$\displaystyle \frac{\partial }{\partial t}F(t;s)=f^{\ast
}(s)\frac{\partial }{\partial s} F(t;s),\ \ \ F(0;s)=s,$ where
$f^{\ast }(s)=\mu (f(s)-s)$ is the infinitesimal generating
function. Therefore
\begin{eqnarray*}
\frac{d}{ds}\Delta (s) &=&-\frac{1}{f^{\ast }(s)}\int_{0}^{\infty
}\frac{d}{dF}g(F(t;s))\frac{\partial }{\partial t}F(t;s)dt
=-\frac{1}{f^{\ast }(s)}\int_{0}^{\infty }d_{t}g(F(t;s)) \\
&=&-\frac{1}{f^{\ast }(s)}[g(F(\infty
;s))-g(F(0;s))]=-\frac{1}{f^{\ast }(s)}[1-g(s)].
\end{eqnarray*}
Hence by integrating from $s$ to $1$ and using that\ $\Delta (1)=0$
one obtains  the first equation in (\ref{QDeltas}). On the other
hand, by integrating form $0$ to $s$ one obtains $\displaystyle
\Delta (s)-\Delta (0)=-\int_{0}^{s}\frac{1-g(u)}{f^{\ast }(u)}du.$
Since $\Delta (0)=Q$ then  the second equation in (\ref{QDeltas})
follows.
\end{proof}
\begin{lemma}
\label{lem1} For $s\in [0,1)$ it is fulfilled that $%
q(t;s)/q(t)\rightarrow 1$ as $t\rightarrow \infty .$
\end{lemma}

\begin{proof} From (\ref{qts}) we have
\begin{eqnarray*}
q(t;s)=\frac{1}{\Psi \left( W\left( \mu t+V\left(
\frac{1}{1-s}\right)
\right) \right) }=\frac{1}{\Psi \left( W\left( \mu t\left( 1+\frac{1}{\mu t}%
V\left( \frac{1}{1-s}\right) \right) \right) \right) }.
\end{eqnarray*}%
Since, $s$ is fixed then for every $t$ large enough $1+\frac{1}{\mu t}%
V\left( \frac{1}{1-s}\right) \in [1,2]$. Therefore, by the uniform
convergence of regularly varying functions on compact sets (see
\cite{bgt}, Theorem 1.5.2), it follows that $\displaystyle \Psi
\left( W\left( \mu t\left( 1+V\left( 1/(1-s)\right) /\mu t\right)
\right) \right) \sim \Psi (W(\mu t)),t\rightarrow \infty, $ which
together with (\ref{qts}) and (\ref{gFt}) completes the proof.
\end{proof}

\begin{lemma}
\label{lem2} Let $s(t)=\exp(-\lambda /W(\mu t)),$  $\lambda >0.$
Then for $c>0$,
\begin{eqnarray*}
&& q(ct;s(t))/q(t)\rightarrow (c+\lambda ^{-\gamma })^{-\alpha
/\gamma }, \ t\rightarrow \infty.
\end{eqnarray*}
\end{lemma}
\begin{proof} Note first that $\displaystyle1-s(t)\sim \lambda /W(\mu t),\
t\rightarrow \infty .$ Since $V$ and $W$ are inverse to each other
then $\displaystyle V\left( [1-s(t)]^{-1}\right) \sim V(\lambda
^{-1}W(\mu t))\sim \lambda ^{-\gamma }\mu t,\ \ t\rightarrow
\infty.$ Now from (\ref{qts}) we have as $t\rightarrow \infty $
\begin{eqnarray*}
&&q(ct;s(t))=1/\Psi \left( W\left( c\mu t+V\left(
[1-s(t)]^{-1}\right)
\right) \right) \\
&\sim &1/\Psi \left( W\left( c\mu t+\lambda ^{-\gamma }\mu t\right)
\right) \sim 1/\Psi \left( W\left( \mu t)(c+\lambda ^{-\gamma
}\right) ^{1/\gamma
}\right) \\
&\sim &\frac{1}{\Psi (W( \mu t))}\left(c+\lambda ^{-\gamma
}\right)^{-\alpha /\gamma}=q(t)(c+\lambda ^{-\gamma })^{-\alpha
/\gamma },
\end{eqnarray*}%
because of (\ref{gFt}) and the fact that $W(t)$ and $\Psi (t)$ are
r.v.f. with exponents $1/\gamma $ and $\alpha $ respectively.
\end{proof}

\section{Asymptotic behavior of the probability for non-visiting the state
zero}

\label{sec4}

Let us denote
\begin{eqnarray}
\label{yanev1} &&P_t:=\mathbf{P}\{Y(t)>0\}=1-\Phi(t;0)
=1-\exp(-I(t)),
\end{eqnarray}
where $\displaystyle I(t)=\int_{0}^{t}r(t-u)q(u)du$ (see (\ref{qts})
with $s=0$ and (\ref{gFt})).


\begin{theorem}
\label{thm4-1} Let conditions (\ref{rt-0}), (\ref{infinite-var}) ,
and (\ref{im-fin}) hold. Assume also that in (\ref{rt-0}) $\theta
\geq 1$ and $\alpha/\gamma \geq 1$. Then
\begin{eqnarray}
&&\label{dt-asimp0} \mathbf{P}\{Y(t)>0\} \sim  R(t)q(t)+ Q(t)r(t),\
\   \ t \to \infty.
\end{eqnarray}
\end{theorem}

\begin{proof} Note that $R(t)$ and $Q(t)$ are s.v.f. at infinity, $tq(t)=o(Q(t))$
and $tr(t)=o(R(t))$ as $t\rightarrow \infty $ (see \cite{bgt},
Proposition 1.5.8 Eq (1.5.8)). Let $\delta \in (0,1/2)$ be fixed.
Then we have
\begin{eqnarray*}
I(t)=\int_{0}^{t}r(t-u)q(u)du=\int_{0}^{t\delta }+\int_{t\delta
}^{t(1-\delta )}+\int_{t(1-\delta )}^{t}=I_{1}(t)+I_{2}(t)+I_{3}(t).
\end{eqnarray*}%
Note first that
\begin{eqnarray*}
&&I_{1}(t)\leq r(t(1-\delta ))\int_{0}^{t\delta }q(u)du\sim
r(t)(1-\delta
)^{-1}Q(t), \\
&&I_{1}(t)\geq r(t)\int_{0}^{t\delta }q(u)du\sim r(t).Q(t),\
t\rightarrow \infty .
\end{eqnarray*}
Then
\begin{eqnarray}
&&\label{I1-1} 1 \le \liminf_{t \to \infty} \frac{I_1(t)}{r(t).Q(t)}
\le \limsup_{t \to \infty}\frac{I_1(t)}{r(t).Q(t)} \le
(1-\delta)^{-\theta}.
\end{eqnarray}
 Since
\begin{eqnarray*}
I_{3}(t) &\leq &q(t(1-\delta ))\int_{0}^{t\delta }r(u)du\sim
q(t)(1-\delta
)^{-\alpha /\gamma }R(t), \\
&&I_{3}(t)\geq q(t)\int_{0}^{t\delta }r(u)du\sim q(t).R(t),\
t\rightarrow \infty ,
\end{eqnarray*}%
then
\begin{eqnarray}
&&\label{I3-1} 1 \le \liminf_{t \to \infty} \frac{I_3(t)}{q(t).R(t)}
\le \limsup_{t \to \infty}\frac{I_3(t)}{q(t).R(t)} \le
(1-\delta)^{-\alpha/\gamma}.
\end{eqnarray} Finally, for $I_2(t)$ we have
\begin{eqnarray*}
0 &\le& I_2(t)=\int_{t \delta}^{t(1-\delta)}r(t-u)q(u)du \\
&\le&
 r(t\delta)q(t\delta)t(1-2\delta) \sim t.r(t).q(t)\delta^{-2}(1-2 \delta), t \to \infty. \nonumber
\end{eqnarray*}
Since $tr(t)=o(R(t))$ and $tq(t) =o(Q(t))$, as $t \to \infty$ then
\begin{eqnarray}
&& \label{I2-1} I_2(t)=o(q(t).R(t)), \ \ \ I_2(t)=o(r(t).Q(t)), \ t
\to \infty.
\end{eqnarray}
Having in mind that $\delta \in (0,1/2)$ was arbitrary,   we
conclude from(\ref{I1-1}), (\ref{I3-1}), and (\ref{I2-1})
\begin{eqnarray*}
&& I(t) =I_1(t)+I_2(t)+I_3(t) \sim r(t).Q(t)+q(t).R(t) \to 0, \ t
\to \infty,
\end{eqnarray*}
which together with (\ref{yanev1}) and $1-e^{-x} = x(1+o(1)), \ \ x
\to 0$ proves (\ref{dt-asimp0}).
\end{proof}

\begin{com}
\label{com41} If $\theta >\frac{\alpha }{\gamma }\geq 1$ then
$R<\infty $
and $P_{t}\sim RL_{Q}(t)t^{-\alpha /\gamma }.$ If $\frac{\alpha }{\gamma}%
>\theta \geq 1$ then $Q<\infty $ and $P_{t}\sim QL_{R}(t)t^{-\theta }.$ If $%
\theta = \frac{\alpha }{\gamma }>1$ then $P_{t}\sim
[RL_{Q}(t)+QL_{R}(t)]t^{-\theta }.$ If $\theta =\frac{\alpha
}{\gamma }=1$ then $R(t)$ and $Q(t)$ are s.v.f. and $P_{t}\sim
\lbrack R(t)L_{Q}(t)+Q(t)L_{R}(t)]t^{-1}.$
\end{com}


\begin{theorem}
\label{thm4-2} Assume the conditions(\ref{rt-0}), (\ref{infinite-var}), and (%
\ref{im-fin}) hold.

\bigskip
(i) If $\theta \geq 1$ and $0<\alpha /\gamma <1$ then
$\mathbf{P}\{Y(t)>0\}\sim R(t).q(t),\ t\rightarrow \infty .$

\bigskip
(ii) If $0<\theta <1$ and $\alpha /\gamma \geq 1$ then
$\mathbf{P}\{Y(t)>0\}\sim r(t).Q(t),t\rightarrow \infty .$
\end{theorem}

\begin{proof} (i) Under the conditions of this case $R(t)=\int_{0}^{t}r(u)du%
\uparrow R\leq \infty $, $R(t)$ is a s.v.f. at infinity and $%
tr(t)=o(R(t)),t\rightarrow \infty $. Let $\delta \in (0,1)$ be
fixed.

Consider $\displaystyle
I(t)=\int_{0}^{t}r(t-u)q(u)du=\int_{0}^{t\delta }+\int_{t\delta
}^{t}=I_{1}(t)+I_{2}(t). $ Then
\begin{eqnarray*}
&&I_{2}(t)=\int_{t\delta }^{t}r(t-u)q(u)du\leq q(t\delta
)\int_{0}^{t(1-\delta )}r(u)du\sim q(t)\delta ^{-\alpha /\gamma }R(t), \\
&&I_{2}(t)\geq q(t)R(t(1-\delta ))\sim q(t)R(t),t\rightarrow \infty .\\
&&0\leq I_{1}(t)=\int_{0}^{t\delta }r(t-u)q(u)du\leq r(t(1-\delta))\int_{0}^{t\delta }q(u)du \\
&\sim &r(t(1-\delta ))\frac{t\delta q(t\delta )}{1-\alpha /\gamma
}\sim \frac{tr(t)q(t)}{1-\alpha /\gamma }(1-\delta )^{-\theta
}\delta ^{1-\alpha/\gamma }.
\end{eqnarray*}
As we mentioned above $tr(t)=o(R(t)),t\rightarrow \infty $.
Therefore $I_{1}(t)=o(I_{2}(t)))$ and
\begin{eqnarray*}
1\leq \liminf_{t\rightarrow \infty }I(t)/(q(t)R(t)) \leq
\limsup_{t\rightarrow \infty }I(t)/(q(t)R(t)) \leq \delta
^{-\alpha/\gamma}.
\end{eqnarray*}
Since $\delta \in (0,1)$ was arbitrary then we get that $I(t)\sim
q(t)R(t)\rightarrow 0,t\rightarrow \infty $. By (\ref{yanev1}) and
$1-e^{-x}\sim x,\ \ x\rightarrow 0,$ we complete the proof of case
(i). The proof of case (ii) is similar, one has to change only the
role of $r(t)$ and $q(t)$.
\end{proof}

\begin{com}
\label{com42} (i) If $\theta \geq 1$ then $R(t)$ is a s.v.f. and
$P_{t}\sim R(t)L_{Q}(t)t^{-\alpha /\gamma }.$

(ii)If $\alpha /\gamma \geq 1$ then $Q(t)$ is a s.v.f. and
$P_{t}\sim  Q(t)L_{R}(t)t^{-\theta }.$
\end{com}


\begin{theorem}
\label{thm4-3} Assume conditions(\ref{rt-0}), (\ref{infinite-var}),
and (\ref{im-fin}) hold. Let additionally $0<\theta <1$ and
$0<\alpha /\gamma <1.$

(i) If $\theta +\alpha /\gamma >1,$ or $\theta +\alpha /\gamma =1$
but $L_{R}(t)L_{Q}(t)\rightarrow 0$ then
$$\mathbf{P}\{Y(t)>0\}\sim t.r(t).q(t).\mathbb{B}(1-\alpha /\gamma ,1-\theta),$$
 where $\mathbb{B}(.,.)$ is Euler's beta function.

(ii) If $\theta +\alpha /\gamma =1$ but $L_{R}(t)L_{Q}(t)\rightarrow
K\in (0,\infty )$ then
$$\mathbf{P}\{Y(t)>0\}\rightarrow 1-e^{-K\pi /\sin \pi \theta },\ t \rightarrow \infty .$$

(iii) If $\theta +\alpha /\gamma <1,$ or $\theta +\alpha /\gamma
=1$\ but $L_{R}(t)L_{Q}(t)\rightarrow \infty$ then
$$\mathbf{P}\{Y(t)>0\}\rightarrow 1,\ t \rightarrow \infty .$$
\end{theorem}

\begin{proof} Let $\delta \in (0,1/2)$ be fixed. Consider
\begin{eqnarray*}
\displaystyle I(t)=\int_{0}^{t\delta }+\int_{t\delta }^{t(1-\delta
)}+\int_{t(1-\delta )}^{t}=I_{1}(t)+I_{2}(t)+I_{3}(t).
\end{eqnarray*}
Changing variables $u=vt$ in $I_2(t)$ we obtain $\displaystyle
I_{2}(t)=t\int_{\delta }^{(1-\delta )}r(t(1-v))q(tv)dv.$ By the
uniform convergence of r.v.f. on compact sets we have for $v\in
[\delta ,1-\delta ]$,
\begin{eqnarray*}
&& \displaystyle r(t(1-v))\sim r(t)(1-v)^{-\theta },\ \ \ q(tv)\sim
q(t)v^{-\alpha /\gamma },t\rightarrow \infty.
\end{eqnarray*}
Therefore, as $t \to \infty$,
\begin{eqnarray*}\label{thm43-i1}
I_{2}(t) = tr(t)q(t)\int_{\delta }^{(1-\delta )}\frac{r(t(1-v))q(tv)
dv}{r(t)q(t)} \sim  tr(t)q(t)\int_{\delta }^{(1-\delta
)}(1-v)^{-\theta }v^{-\alpha/\gamma }dv.
\end{eqnarray*}
Further one gets
\begin{eqnarray}
&&0\leq I_{1}(t)=\int_{0}^{t\delta }r(t-u)q(u)du\leq r(t(1-\delta ))\int_{0}^{t\delta }q(u)du  \notag \\
&\sim &r(t)(1-\delta )^{-\theta }\frac{t\delta q(t\delta )}{1-\alpha
/\gamma }
\sim \frac{tr(t)q(t)}{1-\alpha /\gamma }(1-\delta )^{-\theta }\delta^{1-\alpha /\gamma },  \label{thm33I1}\\
&& 0 \leq I_{3}(t)=\int_{t(1-\delta )}^{t}r(t-u)q(u)du\leq q(t(1-\delta ))\int_{0}^{t\delta }r(u)du  \notag \\
&\sim & q(t)(1-\delta )^{-\alpha /\gamma }\frac{t\delta r(t\delta
)}{1-\theta } \sim \frac{tr(t)q(t)}{1-\theta }(1-\delta )^{-\alpha
/\gamma }\delta^{1-\theta },t\rightarrow \infty .  \label{thm33I2}
\end{eqnarray}
Using the estimates for $I_{1}(t),I_{2}(t),$ and $I_{3}(t)$ we
obtain that
\begin{eqnarray*}
&&\int_{\delta }^{(1-\delta )}(1-v)^{-\theta }v^{-\alpha /\gamma }dv \\
&\leq &\liminf_{t\rightarrow \infty
}\frac{I_{1}(t)+I_{2}(t)+I_{3}(t)}{tr(t)q(t)}
\leq \limsup_{t\rightarrow \infty }\frac{I_{1}(t)+I_{2}(t)+I_{3}(t)}{tr(t)q(t)} \\
&\leq &\frac{(1-\delta )^{-\theta }\delta ^{1-\alpha /\gamma
}}{1-\alpha/\gamma } +\int_{\delta }^{(1-\delta )}(1-v)^{-\theta
}v^{-\alpha /\gamma }dv +\frac{(1-\delta )^{-\theta }\delta
^{1-\alpha /\gamma }}{1-\theta }.
\end{eqnarray*}
These inequalities and the fact that $\delta \in (0,1/2)$ was
arbitrary yield
\begin{eqnarray*}
I(t) &\sim & tr(t)q(t)\mathbb{B}(1-\alpha /\gamma ,1-\theta ) \\
&\sim &t^{1-\theta -\alpha /\gamma
}L_{R}(t)L_{Q}(t)\mathbb{B}(1-\alpha /\gamma ,1-\theta ),\ \
t\rightarrow \infty .
\end{eqnarray*}

(i) In this case $I(t)\rightarrow 0$ which together (\ref{yanev1})
and $1-e^{-x}=x(1+o(1)),\ \ x\rightarrow 0,$ completes the proof.

(ii) Now $I(t)\rightarrow K\mathbb{B}(\theta ,1-\theta )=K\pi /\sin
\pi \theta $ which proves this case.

(iii) Since $I(t)\rightarrow \infty $ then by (\ref{yanev1}) the
proof of this case is completed.
\end{proof}

\begin{com}
\label{com43} In case (i) if $\theta +\alpha /\gamma >1$ then
\begin{eqnarray*}
&& P_{t}\sim L_{R}(t)L_{Q}(t)\mathbb{B}(1-\alpha /\gamma ,1-\theta
)t^{-(\theta +\alpha /\gamma -1)}\rightarrow 0,t\rightarrow \infty .
\end{eqnarray*}
Otherwise if $\theta +\alpha /\gamma =1$ then $P_{t}\sim L_{R}(t)L_{Q}(t)\mathbb{B%
}(1-\alpha /\gamma ,1-\theta )\rightarrow 0,t\rightarrow \infty ,$
i.e. the probability of non-visiting zero develops like a s.v.f.
\end{com}


\section{Limit theorems}
\label{sec5} Note that the conditional p.g.f. of $Y(t)|Y(t)>0$ has
the form
\begin{eqnarray}
&& \label{eq19}
\mathbf{E}\left[s^{Y(t)}|Y(t)>0\right]=1-(1-\Phi(t;s))/(1-\Phi(t;0)).
\end{eqnarray} and from (\ref{yanev}) we have
\begin{eqnarray}
&& \label{1minusfits} 1-\Phi (t;s)=1-\exp(-I(t;s)), \\
&&  \label{Its} I(t;s):=\int_{0}^{t}r(t-u)q(u;s)du,
\end{eqnarray}
where $q(t;s)$ is defined in (\ref{qts}).


\begin{theorem}
\label{thm5.1} Assume the conditions (\ref{rt-0}),
(\ref{infinite-var}), and (\ref{im-fin}) hold. Assume also that in
(\ref{rt-0}) $\theta \ge 1$ such
that $R=\int_0^\infty r(t)dt<\infty$ and $\alpha/\gamma \ge 1$ such that $%
Q=\int_0^\infty q(t)dt<\infty.$

\bigskip
(i) If $q(t)=o(r(t))$ then $ \displaystyle
\mathbf{E}[s^{Y(t)}|Y(t)>0]\rightarrow 1-\Delta(s)/Q, \ \
t\rightarrow \infty .$

\bigskip
(ii) If $r(t)=o(q(t))$ then $\displaystyle \mathbf{P}\left\{
Y(t)/W(\mu t)\leq x|Y(t)>0\right\} \rightarrow D(\alpha ,\gamma
;x),$ $x\geq 0,$ where
\begin{eqnarray*}
\int_{0}^{\infty }e^{-\lambda x}\mathrm{d}D(\alpha ,\gamma
;x)=\hat{D}(\alpha ,\gamma ;\lambda ) = 1-\frac{\lambda ^{\alpha
}}{(1+\lambda ^{\gamma })^{\alpha /\gamma }},\ \ \lambda >0.
\end{eqnarray*}

\bigskip
(iii) If $r(t)/q(t)\rightarrow d\in (0,\infty ),t\rightarrow \infty
,$ then as $t \to \infty$,
\begin{eqnarray}
&& \label{thm4.1c1} \mathbf{E}[s^{Y(t)}|Y(t)>0] \to \frac{dQ}{dQ+R}\left( 1-\Delta (s)/Q\right),\\
&& \mathbf{P}\left\{ Y(t)/W(\mu t)\leq x|Y(t)>0\right\} \to
\frac{dQ+RD(\alpha ,\gamma ;x)}{dQ+R}, \ x \geq 0. \label{thm4.1c2}
\end{eqnarray}
\end{theorem}

\begin{remark}
From Proposition \ref{proposition} one obtains for the limiting
p.g.f.
\begin{eqnarray*}
&& \varphi(s)=1-\frac{\Delta(s)}{Q}=\left(\int_{0}^{s}\frac{1-g(u)}{f(u)-u}%
du\right)/\left(\int_{0}^{1}\frac{1-g(u)}{f(u)-u}du\right).
\end{eqnarray*}
\end{remark}

\begin{proof} Notice first that under the conditions of the theorem Eq.(\ref{dt-asimp0}) gets the form
\begin{eqnarray}
&& \mathbf{P}\{Y(t)>0\} \sim R.q(t)+ Q.r(t),\ \ \ t \to
\infty.\label{dt-asimp0-a}
\end{eqnarray}
By  Lemma \ref{lem1} we conclude that $\displaystyle Q=\int_0^\infty
q(t)dt<\infty$ yields $\displaystyle \Delta(s)=\int_0^\infty
q(t;s)dt<\infty$ and by the dominated convergence theorem $\Delta(s)
\to Q$ as $s \downarrow 0 $. In this way for every fixed $s\in
[0,1),$ $\displaystyle\frac{q(t,s)}{\Delta (s)}$ is a density on
$[0,\infty )$. Under the conditions of the theorem one has also that
$r(t)/R$ is a density on $[0,\infty )$. Therefore by (Theorem 1,
\cite{bgo}) for any fixed $s\in [0,1)$,
\begin{eqnarray}
&&\label{Its0} I(t;s) \sim \Delta(s).r(t)+q(t;s).R  \to 0, \ \ \ t
\to \infty.
\end{eqnarray}
By the relation $1-e^{-x}\sim x, \ \ x \to 0$ we get that for any
fixed $s \in [0,1)$
\begin{eqnarray}\label{Its1}
&& 1-\Phi(t;s) \sim I(t;s), \ \  t \to \infty.
\end{eqnarray}

(i) Using equations (\ref{dt-asimp0-a}), (\ref{Its0}), and
(\ref{Its1}) we get
\begin{eqnarray*}
\frac{1-\Phi (t;s)}{1-\Phi (t;0)}\sim \frac{\Delta
(s)r(t)+q(t;s).R}{Q.r(t)+q(t).R} =\frac{\Delta
(s)+q(t;s)\frac{R}{r(t)}}{Q+q(t)\frac{R}{r(t)}},t\rightarrow \infty
.
\end{eqnarray*}
In this case we have $q(t)R/r(t)\rightarrow 0,\ q(t;s)R/r(t)\sim
q(t)R/r(t)\rightarrow 0.$ Therefore,
\begin{eqnarray*}
[1-\Phi (t;s)]/[1-\Phi (t;0)]\rightarrow \Delta (s)/Q, \
t\rightarrow \infty,
\end{eqnarray*}
which together with (\ref{eq19}) completes the proof of this case.

(ii) In this case $r(t)/q(t;s)\rightarrow 0,\ t\rightarrow \infty,$
for every fixed $s\in [0,1).$ So,
\begin{eqnarray*}
\frac{1-\Phi (t;s)}{1-\Phi (t;0)}\sim
\frac{q(t;s)}{q(t)}\frac{\Delta (s)
\frac{r(t)}{q(t;s)}+R}{Q\frac{r(t)}{q(t)}+R}\rightarrow 1, \
t\rightarrow \infty.
\end{eqnarray*}
In other words, almost all non-degenerate sample paths go to
infinity. So we need an appropriate normalization in order to get a
proper limit distribution. Let now $s(t)=\exp (-\lambda /W(\mu t)).$
For $\delta \in (0,1)$ fixed one has
\begin{eqnarray*}
I(t;s(t))=\int_{0}^{t\delta
}+\int_{t\delta}^{t}=I_{1}(t;s(t))+I_{2}(t;s(t)).
\end{eqnarray*}
Having in mind that $q(t;s)$ is non-increasing in $t\geq 0$ we have
\begin{eqnarray*}
&&q(t;s(t))\int_{0}^{t\delta}r(u)du \leq I_{1}(t;s(t))=\int_{0}^{t\delta }r(u)q(t-u;s(t))du \\
&\leq &q(t(1-\delta );s(t))\int_{0}^{t\delta}r(u)du.
\end{eqnarray*}%
Using that $\displaystyle \lim_{t \to \infty}\int_0^{t\delta}r(u)du
= R<\infty$ and  applying Lemma \ref{lem2} with $c=1$ and
$c=1-\delta $ one obtains
\begin{eqnarray*}
&& q(t;s(t))\int_0^{t\delta} r(u)du \sim q(t)(1+\lambda^{-\gamma})^{-\alpha /\gamma }R,\\
&& q(t(1-\delta );s(t))\int_0^{t\delta} r(u)du \sim q(t)\left(
1-\delta +\lambda ^{-\gamma }\right)^{-\alpha /\gamma }R,\
t\rightarrow \infty,
\end{eqnarray*}
Hence
\begin{eqnarray*}
R(1+\lambda ^{-\gamma})^{-\frac{\alpha}{\gamma}} \le
\liminf_{t\rightarrow \infty }I_{1}(t;s(t))/q(t) \leq
 \limsup_{t\rightarrow \infty }I_{1}(t;s(t))/q(t)\leq R(1-\delta+\lambda ^{-\gamma })^{-\frac{\alpha}{\gamma} }.
\end{eqnarray*}
On the other hand for $I_{2}(t;s(t))$ we obtain
\begin{eqnarray*}
&&I_{2}(t;s(t))=\int_{t\delta }^{t}r(u)q(t-u;s(t))du\leq r(t\delta)\int_{0}^{t(1-\delta )}q(u;s(t))du \\
&\leq &r(t\delta )\int_{0}^{t(1-\delta )}q(u)du\leq
r(t\delta)Q=o(q(t)),t\rightarrow \infty .
\end{eqnarray*}
From the relations for $I_{1}(t;s(t))$ and $I_{2}(t;s(t))$ we get
\begin{eqnarray*}
\frac{R\lambda ^{\alpha }}{(1+\lambda ^{\gamma })^{\alpha /\gamma }}
\leq \liminf_{t\rightarrow \infty }\frac{I(t;s(t))}{q(t)} \leq
\limsup_{t\rightarrow \infty }\frac{I(t;s(t))}{q(t)}\leq
\frac{R\lambda^{\alpha }}{(1-\delta +\lambda ^{\gamma })^{\alpha
/\gamma }}.
\end{eqnarray*}
Hence $\displaystyle I(t;s(t))\sim q(t)\frac{R\lambda ^{\alpha
}}{(1+\lambda ^{\gamma })^{\alpha /\gamma }}\rightarrow 0,$ which
gives that
\begin{eqnarray*}
1-\Phi (t;s(t))\sim q(t)R\lambda ^{\alpha }(1+\lambda ^{\gamma
})^{-\alpha/\gamma },t\rightarrow \infty,
\end{eqnarray*}
using the asymptotic $1-e^{-x}\sim x,\ \ x\rightarrow 0.$ This
relation, (\ref{eq19}), and (\ref{dt-asimp0-a}) with $r(t)=o(q(t))$
yield
\begin{eqnarray*}
\lim_{t\rightarrow \infty }\mathbf{E}\left[ e^{-\lambda Y(t)/W(\mu
t)}|Y(t)>0\right]
 =\hat{D}_{\gamma }(\lambda )=1-\lambda ^{\alpha }(1+\lambda ^{\gamma})^{-\alpha /\gamma },\ \ \lambda >0,
\end{eqnarray*}
which completes the proof of case (ii).

(iii) We have from (\ref{dt-asimp0}) that
\begin{eqnarray}
\label{fit0casec} && 1-\Phi(t;0) \sim (R+dQ)q(t), t \to \infty,
\end{eqnarray}
From equations (\ref{dt-asimp0-a}), (\ref{Its0}), and (\ref{Its1})
it follows that
\begin{eqnarray*}
1-\Phi (t;s)\sim (R+d\Delta (s)))q(t;s),t\rightarrow \infty .
\end{eqnarray*}
This relation and (\ref{fit0casec}) yield
\begin{eqnarray*}
[1-\Phi (t;s)]/[1-\Phi (t;0)]\rightarrow [R+d\Delta (s)]/[R+dQ],
\end{eqnarray*}
which is equivalent to (\ref{thm4.1c1}).

The obtained discrete limiting distribution is not proper. It has
mass at infinity $\displaystyle \frac{dQ}{R+dQ}$. In other words
there are sample paths that grow very fast and they have to be
normalized by some factor in order to obtain a
proper limiting distribution. Let $s(t)=\exp (-\lambda /W(\mu t))$. For $%
\delta \in (0,1)$ one has
\begin{eqnarray*}
I(t;s(t))=\int_{0}^{t/\delta }+\int_{t\delta
}^{t}=I_{1}(t;s(t))+I_{2}(t;s(t))
\end{eqnarray*}%
and hence
\begin{eqnarray*}
q(t;s(t))\int_{0}^{t\delta }r(u)du\leq I_{1}(t;s(t))\leq q\left(
t(1-\delta );s(t)\right) \int_{0}^{t\delta }r(u)du.
\end{eqnarray*}%
Applying Lemma \ref{lem2} with $c=1$ and $c=1-\delta $ we obtain
\begin{eqnarray*}
\left( 1+\frac{1}{\lambda ^{\gamma }}\right) ^{-\alpha /\gamma }\leq
\liminf_{t\rightarrow \infty }\frac{I_{1}(t;s(t))}{Rq(t)}\leq
\limsup_{t\rightarrow \infty }\frac{I_{1}(t;s(t))}{Rq(t)}\leq
 \left(1-\delta +\frac{1}{\lambda^{\gamma }}\right)^{-\alpha/\gamma}.
\end{eqnarray*}%
Since $q(t)$ is non-increasing in $s\in [0,1)$ then one has that
\begin{eqnarray*}
I_{2}(t;s(t)\leq r(t\delta )\int_{0}^{t(1-\delta )}q(u;s(t))du\leq
r(t\delta )q(t)t(1-\delta )=o(q(t)),
\end{eqnarray*}%
having in mind that $tr(t)=o(1),\ t\rightarrow \infty $. Therefore
\begin{eqnarray*}
\left( 1+\frac{1}{\lambda ^{\gamma }}\right) ^{-\alpha /\gamma }\leq
\liminf_{t\rightarrow \infty }\frac{I(t;s(t))}{Rq(t)}\leq
\limsup_{t\rightarrow \infty }\frac{I(t;s(t))}{Rq(t)}\leq \left( 1-\delta +%
\frac{1}{\lambda ^{\gamma }}\right) ^{-\alpha /\gamma }.
\end{eqnarray*}
Since $\delta \in (0,1)$ was arbitrary we get $I(t;s(t)) \sim
R.q(t).\left(1+1/\lambda ^{\gamma }\right) ^{-\alpha /\gamma
}\rightarrow 0,\ t\rightarrow \infty .$ Therefore,
\begin{eqnarray*}
1-\Phi ((t;s(t))\sim R.q(t).\left( 1+1/\lambda ^{\gamma }\right)
^{-\alpha/\gamma }\rightarrow 0,\ t\rightarrow \infty,
\end{eqnarray*}
which leads to
\begin{eqnarray*}
\lim_{t\rightarrow \infty }\frac{1-\Phi (t;s(t))}{1-\Phi (t;0)}=\frac{R}{R+dQ%
}\left( 1+\frac{1}{\lambda ^{\gamma }}\right) ^{-\alpha /\gamma }.
\end{eqnarray*}%
This relation and (\ref{fit0casec}) prove (\ref{thm4.1c2}).
\end{proof}

\begin{com}
\label{com51} (i) Since $q(t)/r(t)=t^{-(\alpha /\gamma -\theta
)}L_{Q}(t)/L_{R}(t)\rightarrow 0$ then $\alpha /\gamma >\theta $ or
$\alpha /\gamma =\theta $ but $L_{Q}(t)/L_{R}(t)\rightarrow 0.$ Let
us consider a particular case when in (\ref{infinite-var})
$L(s)\equiv 1/(1+\gamma )$ and in (\ref{im-fin}) $l(s)\equiv 1.$
Then for $\alpha >\gamma $ we obtain
\begin{eqnarray*}
\int_{0}^{s}\frac{1-g(x)}{f(x)-x}dx=\frac{1+\gamma }{\alpha -\gamma }%
(1-(1-s)^{\alpha -\gamma }.
\end{eqnarray*}%
Therefore $Q=\int_{0}^{1}\frac{1-g(x)}{f(x)-x}dx=\frac{1+\gamma
}{\alpha -\gamma }$ and
\begin{eqnarray*}
\varphi
(s)=\frac{1}{Q}\int_{0}^{s}\frac{1-g(x)}{f(x)-x}dx=1-(1-s)^{\alpha
-\gamma }.
\end{eqnarray*}%
Hence the limiting r.v. belongs to the normal domain of attraction
of a stable law with parameter $\alpha -\gamma .$ Note that for
$\alpha =\gamma $ we have
$\int_{0}^{s}\frac{1-g(x)}{f(x)-x}dx=(1+\gamma )\log \frac{1}{1-s}$
and therefore $Q=\infty .$

(ii) By the Tauberian theorem (Feller \cite{feler}, Ch. XIII,
(5.20)) one has $1-D(\alpha ,\gamma ;x)\sim x^{-\alpha }/\Gamma
(1-\alpha ),x\rightarrow \infty .$ Note that $\alpha /\gamma <\theta
$ or $\alpha /\gamma =\theta $ but $L_{R}(t)/L_{Q}(t)\rightarrow 0.$

(iii) We obtained (with different normalization) two singular to
each other conditional limiting distributions. The first one is a
discrete non-proper distribution similar to the case (i) but now
with an atom at infinity with probability $R/(dQ+R).$ The second one
is similar to the case (ii) \ but it
has now an atom at zero with probability $dQ/(dQ+R).$ In other words, as $%
t\rightarrow \infty $ it follows for the sample paths that $Y(t)\sim
\zeta _{1}$ with probability $dQ/(dQ+R)$ where $\zeta _{1}$ is a
limiting r.v. in the first case and $Y(t)\sim \zeta _{2}W(\mu t)$
with probability $R/(dQ+R)$
where $\zeta _{2}$ is the limiting r.v. in the second case. Note that $%
\alpha /\gamma =\theta $ and $L_{R}(t)/L_{Q}(t)\rightarrow d.$
\end{com}


\begin{theorem}
\label{thm5.2} Assume the conditions (\ref{rt-0}),
(\ref{infinite-var}), and (\ref{im-fin}) hold.

\bigskip
(i) If $\theta =1$ such that $R(t)\rightarrow \infty ,t\rightarrow
\infty ,$ and $\alpha /\gamma >1$ then
\begin{eqnarray*}
&& \mathbf{E}[s^{Y(t)}|Y(t)>0] \to 1-\Delta (s)/Q.
\end{eqnarray*}

\bigskip
(ii) If $\theta >1$ and $\alpha /\gamma =1$ such that
$Q(t)\rightarrow \infty, t\rightarrow \infty ,$ then
\begin{eqnarray*}
&& \mathbf{P}\left\{ Y(t)/W(\mu t)\leq x|Y(t)>0\right\} \rightarrow
D(\alpha ,\gamma ;x), x\geq 0,
\end{eqnarray*}
 where $D(\alpha ,\gamma ;x),x\geq 0$  is defined in Theorem 5.1 (ii).
\end{theorem}

\begin{proof} (i) In this case $q(t,s)=o(r(t)),t\rightarrow \infty ,$ for any fixed
$s\in [0,1)$ and $1-\Phi (t;s)\sim r(t).\Delta (s).$ Following the
same way as in the proof of Theorem \ref{thm5.1} (i) we can complete
now the proof.

(ii) In this case $\displaystyle R=\int_{0}^{\infty }r(t)dt<\infty $ and $%
r(t)=o(q(t)),t\rightarrow \infty $. Using that
$q(t;s)/q(t)\rightarrow
1,t\rightarrow \infty $ for any fixed $s\in [0,1)$ (see Lemma \ref%
{lem1}) we can obtain similarly as in the proof of Theorem
\ref{thm5.1} (ii) that $1-\Phi (t;s)\sim R.q(t;s),t\rightarrow
\infty ,\ \mathrm{\ for\ any\ fixed}\ s\in [0,1).$ Therefore
$[1-\Phi (t;s)]/[1-\Phi (t;0)]\rightarrow 1,t\rightarrow \infty ,$
i.e. almost all non-degenerate sample paths go to infinity. Working
in the same way as in the proof of Theorem \ref{thm5.1} (ii) we are
able to complete the proof.
\end{proof}

\begin{com}
\label{com52} Theorem \ref{thm5.2} can be interpreted as an
extension of the cases (i) and (ii) of Theorem \ref{thm5.1}, where
both $R$ and $Q$ are finite. Now $Q(t)\rightarrow \infty $ in (i)
and $R(t)\rightarrow \infty $ in (ii).
\end{com}


\begin{theorem}
\label{thm5.3} Assume the conditions (\ref{rt-0}),
(\ref{infinite-var}), and
(\ref{im-fin}) hold. If $\theta =1$ and $\alpha =\gamma $ such that $%
R(t)\uparrow R\leq \infty, \ Q(t)\uparrow \infty,$ and
 $\displaystyle \frac{r(t)Q(t)}{q(t)R(t)}\rightarrow d \geq 0, t \rightarrow \infty$ then
\begin{eqnarray*}
\mathbf{P}\left\{ Y(t)/W(\mu t)\leq x|Y(t)>0\right\} \rightarrow
D_{\gamma ,d}(x),x\geq 0,
\end{eqnarray*}%
where $D_{\gamma }(x)$ has Laplace transform
\begin{eqnarray*}
\hat{D}_{\gamma ,d}(\lambda
)=\frac{d}{1+d}+\frac{1}{1+d}.\frac{1}{1+\lambda^{\gamma }},\lambda
>0.
\end{eqnarray*}
\end{theorem}

\begin{proof} Let $s(t)=\exp (-\lambda /W(\mu t))$ and for $\delta \in (0,1)$ fixed
one has
\begin{eqnarray*}
I(t;s(t))=\int_{0}^{t\delta }+\int_{t\delta
}^{t}=I_{1}(t;s(t))+I_{2}(t;s(t)).
\end{eqnarray*}
Since $q(t,s)$ is non-increasing in $t$ we get that
\begin{eqnarray*}
q(t;s(t))\int_{0}^{t(1-\delta )}r(u)du\leq I_{2}(t;s(t))\leq
q(t\delta ;s(t))\int_{0}^{t(1-\delta )}r(u)du.
\end{eqnarray*}%
Applying Lemma \ref{lem2} with $c=\delta $ and $c=1$ respectively,
one obtains
\begin{eqnarray*}
q(t\delta ;s(t))\sim \frac{q(t)}{\delta +\lambda ^{-\gamma }},\ \ \
\ q(t;s(t))\sim \frac{q(t)}{1+\lambda ^{-\gamma }},\ t\rightarrow
\infty ,
\end{eqnarray*}%
having in mind that $\alpha =\gamma .$ Therefore
\begin{eqnarray}
&&\label{asimpI2} \frac{1}{1+\lambda^{-\gamma}} \le \liminf_{t \to
\infty} \frac{I_2(t;s(t))}{q(t)R(t)}
 \le \limsup_{t \to \infty}\frac{I_2(t;s(t))}{q(t)R(t)} \le \frac{1}{\delta+\lambda^{-\gamma}}.
\end{eqnarray}
Let us consider $I_{1}(t;s(t))$. Since $q(u;s)$ is non-increasing in
$u$ then
\begin{eqnarray*}
q(u;s(t)))\leq q(0;s(t))=\frac{1}{\Psi \left( W\left( V\left( \frac{1}{1-s(t)%
}\right) \right) \right) }=\frac{1}{\Psi \left(
\frac{1}{1-s(t)}\right) },
\end{eqnarray*}
because $W(.)$ and $V(.)$ are inverse to each other. Therefore
\begin{eqnarray*}
&&I_{1}(t;s(t))=\int_{0}^{t\delta }r(t-u)q(u;s(t))du\leq
r(t(1-\delta
))\int_{0}^{t\delta }\frac{1}{\Psi \left( \frac{1}{1-s(t)}\right) }du \\
&\leq &r(t(1-\delta ))\frac{t\delta }{\Psi \left( \frac{1}{1-s(t)}\right) }%
\sim r(t(1-\delta ))\frac{t\delta }{\Psi \left( \frac{W(\mu t)}{\lambda }%
\right) }\sim r(t(1-\delta ))(t\delta )q(t)\lambda ^{\alpha },
\end{eqnarray*}
using the relation $1-\exp (-\lambda /W(\mu t))\sim \lambda /W(\mu
t),t\rightarrow \infty ,$ and the properties of $\Psi (.)$.
Therefore $I_{1}(t;s(t))=O(r(t(1-\delta ))q(t)t),t\rightarrow
\infty.$ Having in mind that $tr(t)=o(R(t))$ we conclude that
$I_{1}(t;s(t))=o(q(t)R(t)),t \rightarrow \infty .$ From here and
(\ref{asimpI2}) we get that
\begin{eqnarray*}
\frac{1}{1+\lambda ^{-\gamma }}\leq
 \liminf_{t\rightarrow \infty }\frac{I(t;s(t))}{q(t)R(t)}\leq
  \limsup_{t\rightarrow \infty }\frac{I(t;s(t))}{q(t)R(t)}
  \leq \frac{1}{\delta +\lambda ^{-\gamma }}.
\end{eqnarray*}%
Since $\delta \in (0,1)$ was arbitrary we conclude that
\begin{eqnarray*}
\lim_{t\rightarrow \infty }I(t;s(t))/[q(t)R(t)]=1/(1+\lambda
^{-\gamma }).
\end{eqnarray*}
Then  $I(t;s(t))\sim q(t)R(t)/[1+\lambda ^{-\gamma }]\rightarrow 0,$
and $1-\Phi (t;s(t))\sim q(t)R(t)/[1+\lambda^{-\gamma }],$ $t \to
\infty.$ Therefore
\begin{eqnarray*}
\frac{1-\Phi (t;s(t))}{1-\Phi (t;0)} \sim \frac{(1+\lambda
^{-\gamma})^{-1}q(t)R(t)}{q(t).R(t)+r(t)Q(t)} \rightarrow
\frac{(1+\lambda ^{-\gamma})^{-1}}{1+d}.
\end{eqnarray*}
which completes the proof of the theorem.
\end{proof}

\begin{com}
\label{com53} The conditional limiting distribution has an atom at
zero with probability $\frac{d}{1+d}.$ Then applying a Tauberian
theorem (as in Comment \ref{com51}) one can obtain that $1-D_{\gamma
,d}(x)\sim x^{-\gamma }/[(1+d)\Gamma (1-\gamma )$, i.e. the limiting
r.v. belongs to a normal
domain of attraction of a stable law with parameter $\gamma .$ Note that $%
\alpha =\gamma $ and $\frac{L_{R}(t)Q(t)}{L_{Q}(t)R(t)}\rightarrow
d. $
\end{com}


\begin{theorem}
\label{thm5.4} Assume the conditions (\ref{rt-0}),
(\ref{infinite-var}), and (\ref{im-fin}) hold.

(i) If $\theta \geq 1$ and $0<\alpha /\gamma <1$ then
\begin{eqnarray*}
\mathbf{P}\left\{ Y(t)/W(\mu t)\leq x|Y(t)>0\right\} \to D(\alpha
,\gamma ;x),x\geq 0.
\end{eqnarray*}
where $D(\alpha ,\gamma ;x),x\geq 0$ is defined in Theorem 5.1(ii).

(ii) If $0<\theta <1$ and $\alpha /\gamma \geq 1$ such that
$\displaystyle Q=\int_{0}^{\infty }q(u)du<\infty $ then
\begin{eqnarray*}
\mathbf{E}\left[ s^{Y(t)}|Y(t)>0\right]  \to 1-\Delta (s)/Q.
\end{eqnarray*}
\end{theorem}

\begin{proof} (i) Under the conditions of this case one has
$R(t)\rightarrow R\leq \infty $, $R(t)$ is a s.v.f. and
$tr(t)=o(R(t)),\ t\rightarrow \infty $. Let $s(t)=1-\exp (-\lambda
/W(\mu t))$ and consider for any fixed $\delta \in (0,1)$
\begin{eqnarray*}
I(t;s(t))=\int_{0}^{t\delta }+\int_{t\delta
}^{t}=I_{1}(t;s(t))+I_{2}(t;s(t)).
\end{eqnarray*}
Hence
\begin{equation}
q(t;s(t))\int_{0}^{t(1-\delta )}r(u)du\leq I_{2}(t;s(t))\leq
q(t\delta ;s(t))\int_{0}^{t(1-\delta )}r(u)du  \label{ineq-i2}.
\end{equation}
Applying Lemma \ref{lem2} with $c=\delta $ and $c=1$ respectively,
one gets
\begin{eqnarray*}
q(t\delta ;s(t))\sim q(t)(\delta +\lambda ^{-\gamma })^{-\alpha
/\gamma },\ q(t;s(t))\sim q(t)(1+\lambda ^{-\gamma })^{-\alpha
/\gamma },t\rightarrow \infty .
\end{eqnarray*}%
These two relations and (\ref{ineq-i2}) provided that
\begin{eqnarray*}
(1+\lambda ^{-\gamma })^{-\alpha /\gamma }\leq \liminf_{t\rightarrow \infty }%
\frac{I_{2}(t;s(t))}{q(t).R(t)}\leq \limsup_{t\rightarrow \infty }\frac{%
I_{2}(t;s(t))}{q(t).R(t)}\leq (\delta +\lambda ^{-\gamma })^{-\alpha
/\gamma }.
\end{eqnarray*}%
On the other hand we have
\begin{eqnarray*}
&&0\leq I_{1}(t;s(t))=\int_{0}^{t\delta }r(t-u)q(u;s(t))du\leq
r(t(1-\delta
))\int_{0}^{t\delta }q(u)du \\
&\sim &r(t)(1-\delta )^{\theta }\frac{t\delta q(t\delta )}{1-\alpha /\gamma }%
\sim \frac{r(t)(1-\delta )^{\theta }tq(t)\delta ^{1-\alpha /\gamma }}{%
1-\alpha /\gamma }=o(R(t).q(t)),t\rightarrow \infty ,
\end{eqnarray*}%
remember that $tr(t)=o(R(t)),t\rightarrow \infty $. Therefore,
\begin{eqnarray*}
(1+\lambda ^{-\gamma })^{-\alpha /\gamma }\leq \liminf_{t\rightarrow \infty }%
\frac{I(t;s(t))}{q(t).R(t)}\leq \limsup_{t\rightarrow \infty }\frac{I(t;s(t))%
}{q(t).R(t)}\leq (\delta +\lambda ^{-\gamma })^{-\alpha /\gamma }.
\end{eqnarray*}%
Since $\delta \in (0,1)$ was arbitrary, we get
\begin{eqnarray*}
I(t;s(t))\sim (1+\lambda ^{-\gamma })^{-\alpha /\gamma
}q(t).R(t)\rightarrow 0,t\rightarrow \infty ,
\end{eqnarray*}%
for any fixed $\lambda >0$. Therefore
\begin{eqnarray*}
1-\Phi (t;s(t))\sim I(t;s(t))\sim (1+\lambda ^{-\gamma })^{-\alpha
/\gamma}(1-\Phi (t;0)).
\end{eqnarray*}%
This relation and (\ref{eq19}) compete the proof of this case.

(ii) Let $\delta \in (0,1)$ be fixed and
\begin{eqnarray}
&&\label{44ii-1} I(t;s)=\int_0^t r(t-u)q(u;s)du=\int_0^{t
\delta}+\int_{t \delta}^t=I_1(t;s)+I_2(t;s).
\end{eqnarray}
Then
\begin{eqnarray*}
r(t)\int_{0}^{t\delta }q(u;s)du\leq I_{1}(t;s)\leq
r(t(1-\delta))\int_{0}^{t\delta }q(u;s)du.
\end{eqnarray*}
Since $Q<\infty $ then by Lemma \ref{lem1} it follows that
$\displaystyle \Delta (s)=\int_{0}^{\infty }q(u;s)du<\infty $ and by
the dominated convergence theorem $\Delta (s)\rightarrow Q$, as
$s\rightarrow 0$. Therefore,
\begin{eqnarray}
&&\label{44ii-2} \frac{ \Delta(s)}{Q} \le \liminf_{t \to \infty}
\frac{I_1(t;s)}{r(t)Q(t)} \limsup_{t \to
\infty}\frac{I_1(t;s)}{r(t)Q(t)}  \le
(1-\delta)^{-\theta}\frac{\Delta(s)}{Q}.
\end{eqnarray}
Having in mind that $q(t;s)$ is non-increasing in $t\geq 0$ and
$s\in [0,1)$ we obtain
\begin{eqnarray*}
I_{2}(t;s)=\int_{t\delta }^{t}r(t-u)q(u;s)du\leq q(t\delta
)\int_{0}^{t}r(u)du\sim q(t\delta )\frac{tr(t)}{1-\theta },
\end{eqnarray*}
because $r(t)$ varies regularly with exponent $-\theta \in (-1,0)$.
Therefore
\begin{eqnarray}
&&\label{44ii-3} \frac{I_2(t;s)}{r(t)Q(t)} \le\frac{
tq(t\delta)r(t)}{(1-\theta)r(t)Q(t)} \to \ 0, t \to \infty,
\end{eqnarray}
because in this case $tq(t)=o(Q(t)),t\rightarrow \infty .$ From
(\ref{44ii-1}), (\ref{44ii-2}), and (\ref{44ii-3}) it follows that
\begin{eqnarray*}
\frac{\Delta (s)}{Q} \leq \liminf_{t\rightarrow \infty
}\frac{I(t;s)}{r(t)Q(t)}\limsup_{t\rightarrow \infty
}\frac{I(t;s)}{r(t)Q(t)} \leq (1-\delta)^{-\theta }\frac{\Delta
(s)}{Q},\ \ s\in [0,1).
\end{eqnarray*}
Since $\delta $ was arbitrary, we get that $I(t;s)\sim
\Delta(s)r(t)Q(t)/Q \rightarrow 0,  t\rightarrow \infty.$ Therefore
$1-\Phi(t;s)\sim \Delta (s)r(t)Q(t)/Q\rightarrow 0,  t\rightarrow
\infty,$ which together with Theorem 4.2(ii) and (\ref{eq19})
completes the proof.
\end{proof}
\begin{com}
\label{com54} Note that the obtained limiting distributions are the
same as in Theorem \ref{thm5.2} (respectively Theorem \ref{thm5.1} -
(ii) and (i)) nevertheless that the conditions and methods of the
proofs are different.
\end{com}


\begin{theorem}
\label{thm5.5} Assume the conditions (\ref{rt-0}),
(\ref{infinite-var}), and (\ref{im-fin}) hold with $0<\theta <1$ and
$0<\alpha /\gamma <1$.\newline (i) If $tr(t)q(t)\rightarrow 0$ then
\begin{eqnarray*}
\mathbf{P}\left\{ Y(t)/W(t)\leq x|Y(t)>0\right\} \to D_{1}(\theta
,\alpha ,\gamma ;x),x\geq 0,
\end{eqnarray*}%
where $D_{1}(\theta ,\alpha ,\gamma ;x)$ has Laplace transform
\begin{eqnarray*}
\hat{D_{1}}(\theta ,\alpha ,\gamma ;\lambda )= 1-\frac{\lambda
^{\alpha }}{B(1-\theta ,1-\frac{\alpha }{\gamma
})}\int_{0}^{1}(1-u)^{-\theta }(u\lambda^{\gamma }+1)^{-\alpha
/\gamma }du.
\end{eqnarray*}
(ii) If $tr(t)q(t)\rightarrow K\in (0,\infty )$ then
\begin{eqnarray*}
\mathbf{P}\left\{ Y(t)/W(t)\leq x|Y(t)>0\right\} \to D_{2}(\theta
,\alpha ,\gamma ;x),x\geq 0,
\end{eqnarray*}
where $D_{2}(\theta,\alpha,\gamma ;x)$ has Laplace transform
\begin{eqnarray*}
\hat{D_{2}}(\theta ,\alpha ,\gamma ;\lambda ) =1-\frac{1-\exp
\left(-K\lambda ^{\alpha }\int_{0}^{1}(1-u)^{-\theta }(u\lambda
^{\gamma}+1)^{1-\theta }du\right)}{1-\exp\left( -K.B(\theta
,1-\theta )\right) }.
\end{eqnarray*}
(iii) If $tr(t)q(t)\rightarrow \infty $ then
\begin{eqnarray*}
\mathbf{P}\left\{ Y(t)/\overleftarrow{\Psi }(W(t))\leq x\right\} \to
D_{3}(\alpha ;x), x \geq 0,
\end{eqnarray*}
where $D_{3}(\theta,\alpha,\gamma ;x)$ has Laplace transform
$\hat{D_{3}}(\alpha ;\lambda)=e^{-\lambda ^{\alpha }},\ \lambda >0.$
\end{theorem}

\begin{proof} (i) Let us denote $s(t)=\exp(-\lambda/W(\mu t))$.
Let $\delta \in (0,\frac{1}{2})$ be fixed. Consider
\begin{eqnarray*}
I(t;s(t))=\int_{0}^{t\delta }+\int_{t\delta
}^{t(1-\delta)}+\int_{t(1-\delta
)}^{t}=I_{1}(t;s(t))+I_{2}(t;s(t))+I_{3}(t;s(t)).
\end{eqnarray*}%
By changing variables $u=vt$ in $I_2(t;s(t))$ we obtain
\begin{eqnarray*}
I_{2}(t;s(t))=\int_{t\delta }^{t(1-\delta
)}r(t-u)q(u;s(t))du=t\int_{\delta}^{1-\delta }r(t(1-v))q(tv;s(t))dv,
\end{eqnarray*}
Applying Lemma \ref{lem2} we have $q(vt;s(t)) \sim q(t)\left(
v+\lambda^{-\gamma }\right)^{-\alpha /\gamma },\ t\rightarrow \infty
.$ Therefore
\begin{eqnarray*}
&&I_{2}(t;s(t))=tr(t)q(t)\int_{\delta }^{1-\delta }\frac{r(t(1-v))}{r(t)}\frac{q(tv;s(t))}{q(t)}dv \\
&\sim &tr(t)q(t)\int_{\delta }^{1-\delta }(1-v)^{-\theta
}(v+\lambda^{-\gamma })^{-\alpha /\gamma }dv\rightarrow 0,\ \
t\rightarrow \infty ,
\end{eqnarray*}
by the uniform convergence of the r.v.f. on compact set $[\delta
,1-\delta ]$. Further
\begin{eqnarray*}
I_{1}(t;s(t)) &=&\int_{0}^{t\delta }r(t-u)q(u;s(t))du\leq \int_{0}^{t\delta }r(t-u)q(u))du \\
I_{3}(t;s(t)) &=&\int_{t(1-\delta )}^{t}r(t-u)q(u;s(t))du\leq
\int_{0}^{t\delta }r(u)q(t-u))du,
\end{eqnarray*}
because $q(t;s)$ is non-increasing in $s\in [0,1]$. By the last
three relations, having also in mind (\ref{thm33I1}) and
(\ref{thm33I2}), and the fact that $\delta $ was arbitrary, we
conclude that
\begin{eqnarray*}
I(t;s(t)\sim tr(t)q(t)\int_{0}^{1}(1-v)^{-\theta }(v+\lambda
^{-\gamma })^{-\alpha /\gamma }dv\rightarrow 0,t\rightarrow \infty .
\end{eqnarray*}%
Therefore $1-\Phi (t;s(t))\sim tr(t)q(t)\int_{0}^{1}(1-v)^{-\theta
}(v+\lambda ^{-\gamma })^{-\alpha /\gamma }dv,t\rightarrow \infty ,$
which together with (\ref{eq19}) and Theorem 4.3(i) completes the
proof of this case.

(ii) Setting as above $s(t)=\exp(-\lambda/W(\mu t))$ and having in
mind that in this case $\theta+ \alpha/\gamma=1$, in the same way as
in the previous case we obtain that as $t \to \infty$,
\begin{eqnarray*}
&& I(t,s(t)) \sim
tq(t)r(t)\lambda^\alpha\int_0^1(1-u)^\theta(u\lambda^\gamma+1)^{1-\theta}du
\\
&\to& K
\lambda^\alpha\int_0^1(1-u)^\theta(u\lambda^\gamma+1)^{1-\theta}du.
\end{eqnarray*}
Then (see (\ref{1minusfits}) and (\ref{Its})),
\begin{eqnarray*}
&& \Phi(t,s(t)) \to \exp\left(-K
\lambda^\alpha\int_0^1(1-u)^\theta(u\lambda^\gamma+1)^{1-\theta}du\right),
\ \ t \to \infty.
\end{eqnarray*}
This limit together with Theorem 4.3(ii) and (\ref{eq19}) completes
the proof of the case (ii).

(iii) Denote by $s(t)=\exp(-\lambda /\overleftarrow{\Psi }(R(t)))$.
Remember
that $\overleftarrow{\Psi }(.)$ is the inverse function of $\Psi(.)$. For $%
\delta \in (0,1)$ we consider
\begin{eqnarray*}
&& I(t;s(t))=\int_{0}^{t}r(t-u)q(u;s(t))du=\int_{0}^{t\delta
}+\int_{t\delta}^{t}=I_{1}(t;s(t))+I_{2}(t;s(t)).
\end{eqnarray*}
Since $q(t;s)$ is non-increasing in $t\geq 0$ we have
\begin{eqnarray*}
&& q(t;s(t))R(t(1-\delta ))\leq I_{2}(t;s(t))\leq q(t\delta
;s(t))R(t).
\end{eqnarray*}
From  (\ref{gFt}) it follows that $\displaystyle 1/q(t) = \Psi
(W(\mu t)).$ Then
 $\displaystyle  \mu t\sim V\left( \overleftarrow{\Psi }\left(1/q(t)\right) \right),$ as
$t \to \infty.$ We have also that  $\displaystyle 1-s(t)\sim
\lambda/\overleftarrow{\Psi }(R(t)).$ Therefore,
\begin{eqnarray*}
\frac{\mu t}{V\left( 1/(1-s(t))\right) }\sim
\frac{V\left(\overleftarrow{\Psi }\left(1/q(t)\right) \right)
}{V\left( \overleftarrow{\Psi }(R(t))/\lambda\right) } \sim \lambda
^{-\gamma }\frac{V\left( \overleftarrow{\Psi }\left( 1/q(t)\right)
\right) }{V\left(\overleftarrow{\Psi }(R(t))\right) } \sim
\frac{\lambda ^{-\gamma }}{q(t)R(t)}\rightarrow 0.
\end{eqnarray*}%
because $\displaystyle q(t)R(t)\sim q(t)\frac{tr(t)}{\theta
+1}\rightarrow \infty $. From this relation (see also (\ref{qts}),
using the uniform convergence of regularly varying functions we get
that as $t\rightarrow \infty$,
\begin{eqnarray*}
&&q(t\delta ;s(t))R(t)=\frac{R(t)}{\Psi \left( W\left( V\left( 1/(1-s(t))\right) \left( \frac{\mu \delta t}{V(1/(1-s(t))}+1\right)\right) \right) } \\
&\sim &\frac{R(t)}{\Psi \left( W\left(
V\left(1/(1-s(t))\right)\right) \right) }
\sim \frac{R(t)}{\Psi (1/(1-s(t)))}\sim \frac{R(t)}{\Psi \left( \overleftarrow{\Psi }(R(t))/\lambda \right) } \\
&\sim &\lambda ^{\alpha }\frac{R(t)}{\Psi \left( \overleftarrow{\Psi
}(R(t))\right) } \sim \lambda ^{\alpha
}R(t).\frac{1}{R(t)}\rightarrow \lambda^{\alpha }.
\end{eqnarray*}
In the same way one has that $q(t\delta ;s(t))R(t(1-\delta
))\rightarrow \lambda ^{\alpha }(1-\delta )$ as $t\rightarrow \infty
.$ Having in mind that $W(.)$ is increasing and $V(x)>0$ for any
$x>1$, we get
\begin{eqnarray*}
&&I_{1}(t;s(t))=\int_{0}^{t\delta }r(t-u)q(u;s(t))du\leq r(t)\int_{0}^{t\delta }q(u;s(t))du \\
&=&r(t)\int_{0}^{t\delta }\frac{du}{\Psi \left( W\left( \mu
u+V\left(1/(1-s(t))\right) \right) \right) }
\leq r(t)\int_{0}^{t\delta }\frac{du}{\Psi \left( W\left( V\left(1/(1-s(t))\right) \right) \right) } \\
&=&\frac{t\delta r(t)}{\Psi \left(1/(1-s(t))\right) }\sim
\frac{t\delta r(t)}{\Psi \left( \overleftarrow{\Psi }(R(t))/\lambda
\right) } \sim \delta (1+\theta )R(t)\frac{\lambda^{\alpha }}{R(t)}
\rightarrow \delta (1+\theta )\lambda ^{\alpha }.
\end{eqnarray*}%
Therefore
\begin{eqnarray*}
(1-\delta )^{\theta +1}\lambda ^{\alpha }\leq \liminf_{t\rightarrow
\infty }I(t;s(t))\leq \limsup_{t\rightarrow \infty }I(t;s(t))\leq
\lambda ^{\alpha }+\delta (1+\theta )\lambda ^{\alpha }.
\end{eqnarray*}%
Since $\delta $ was arbitrary then $\displaystyle \lim_{t\rightarrow
\infty }I(t;s(t))=\lambda ^{\alpha },$ which completes the proof.
\end{proof}

\begin{com}
\label{com55} (i) In this case
$tr(t)q(t)=L_{R}(t)L_{Q}(t)t^{1-\theta -\alpha /\gamma }\rightarrow
0,$ \ which means that $\theta +\alpha /\gamma <1$ or $\theta
+\alpha /\gamma =1$\ but $L_{R}(t)L_{Q}(t)\rightarrow 0.$ The
normalizing function is regularly varying with parameter $1/\gamma $
and we obtain a conditional limiting distribution $D_{1}(x).$ Then
by the Tauberian theorem
\begin{eqnarray*}
&& 1-D_{1}(x)\sim x^{-\alpha }/[\Gamma (1-\alpha )(1-\theta
)B(1-\theta,1-\alpha /\gamma )],\
 x \rightarrow \infty,
\end{eqnarray*}
i.e. the limiting r.v. belongs to a normal domain of attraction of a
stable law with parameter $\alpha .$

(ii) In this case $tr(t)q(t)=L_{R}(t)L_{Q}(t)\rightarrow K\in
(0,\infty ),$ because $\theta +\alpha /\gamma =1.$ Then, with the
same normalizing function as in the previous case we obtain
similarly that
\begin{equation*}
1-D_{2}(x)\sim Kx^{-\alpha }/\{\Gamma (1-\alpha )(1-\theta
)[1-\exp(-K.B(\theta ,1-\theta ))]\}, \ \ x\rightarrow \infty ,
\end{equation*}
(iii) In this case $\theta +\alpha /\gamma <1$\textit{\ or }$\theta
+\alpha /\gamma =1$ but $L_{R}(t)L_{Q}(t)\rightarrow \infty .$ \ The
normalization is by a s.v.f. with parameter $\theta /\alpha $ and
surprisingly the non-conditional limiting distribution is just
stable with parameter $\alpha$.
\end{com}

\section{Concluding remarks}

As it was shown, the asymptotic behavior of the non-visiting zero
probability and the limiting distributions depend of the relations
between parameters of reproduction ($\gamma ),$ of immigration
($\alpha )$ and of the decreasing Poisson intensity ($\theta ),$ as
well as from the corresponding s.v.f. (in some cases). The
probability for non-visiting zero
state converges to zero (with different rates), or to positive constant in $%
(0,1)$ which is exact calculated, or finally to 1. We obtained eight
different limiting distributions under the suitable normalization.
An interesting case is given in Theorem 5.1-(iii) where we obtained
two singular to each other conditional limiting distributions (with
different normalizing functions). Another interesting situation is
presented in Theorem \ref{thm5.5}-(iii) where we obtained
non-conditional limiting distribution which is just stable with
parameter $\alpha.$ Note that the
intensity $r(t)$ can be interpreted as a control function. The case when $%
r(t)$ is increasing is quite different and the obtained results are
accepted for publication in \cite{MitYan2}.

\end{document}